\documentclass[12pt]{amsart}
\usepackage{amsmath, amsthm, amscd, amsfonts}

\newtheorem{theorem}{Theorem}[section]
\newtheorem{lemma}[theorem]{Lemma}
\newtheorem{problem}[theorem]{Problem}
\newtheorem{proposition}[theorem]{Proposition}
\newtheorem{corollary}[theorem]{Corollary}
\theoremstyle{definition}
\newtheorem{definition}[theorem]{Definition}

\theoremstyle{remark}

\newtheorem{remark}[theorem]{Remark}
\numberwithin{equation}{section}

\newfont{\kh}{msbm10}

\begin{document}
\title[invariant submodules
of modular operators]
{invariant submodules
of modular operators and Lomonosov type theorem for Hilbert C*-modules}
\author{K. Sharifi}
\address{Kamran Sharifi, \newline Faculty of Mathematical Sciences, Shahrood University of Technology,
P.\,O.\,Box 3619995161, Shahrood, Iran.} \email{sharifi.kamran@gmail.com}


\begin{abstract}
In this paper, we introduce the notion of invariant submodule in the theory of Hilbert C*-modules
and study some basic properties of bounded adjointable operators and their generalized inverses which
have nontrivial invariant submodules.
We demonstrate the representation of the solution set of an operator equation
on Hilbert C*-modules by taking advantage of invariant submodules. In particular, we
consider the special cases of finite dimensional C*-algebras and C*-algebras of
compact operators as the underling C*-algebra
to simplify our results, and obtain a Lomonosov type theorem for compact operators on
some Hilbert C*-modules.
\end{abstract}
\maketitle

\section{Introduction and preliminary}
In mathematics, an invariant is a property of a mathematical object (or a class of mathematical objects) which remains
without change after operations or transformations of a specific type are applied to the objects. The particular class of
objects and type of transformations are usually indicated by the context in which the term is used.
Invariants are used in diverse areas of mathematics such as geometry, topology, algebra and discrete mathematics.
Some important classes of transformations are defined by an invariant they leave unchanged.

An invariant subspace of a linear mapping $ T:V \to V$ from some vector space $V$ to itself is a subspace $W$ of $V$ such that $T(W)$
is contained in $W$. An invariant subspace of $T$ is also said to be $T$-invariant. If $W$ is $T$-invariant,
we can restrict $T$ to $W$ to arrive at a new linear mapping $T_{|_{W}}: W \to W$.
Clearly $V$ itself, and the subspace $\{0 \}$, are trivially invariant subspaces for each operator
$ T:V \to V$. For certain linear operators there is no non-trivial invariant subspace; consider for instance a rotation of a two-dimensional real vector space.
Let $x$ be an eigenvector of $T$, i.e. $Tx=\lambda x$. Then $W = \mathrm{span} \{x \}$ is $T$-invariant. As a consequence of the fundamental
theorem of algebra, every linear operator on a nonzero finite-dimensional complex vector space has an eigenvector.
Therefore, every such linear operator has a non-trivial invariant subspace.  One can definitely see that the
invariant subspaces of a linear transformation are dependent upon the base field of $V$.
An invariant vector (i.e. a fixed point of $T$), other than $0$, spans an invariant subspace of dimension 1. An invariant subspace
of dimension 1 will be acted on by $T$ by a scalar and consists of invariant vectors if and only if that scalar is 1.
Invariant subspaces have been studied systematically in the books of Radjavi and Rosenthal \cite{RR} and Kubrusly \cite{Kubrusly}.

In the early thirties J. von Neumann proved the existence of proper invariant subspaces for completely continuous operators in a Hilbert space;
his proof was never published.
In 1950, Aronszajn found a proof of the theorem which was verbally verified with von Neumann that it was essentially
the same proof as he had previously found. Then Aronszajn
was able to give a proof for reflexive Banach spaces \cite{Aronszajn}. In particular,
in the Hilbert space setting, the class of completely continuous operators is
identical to the algebra of compact operators.
In 1973, Lomonosov showed a broad generalization of this result  \cite{Lomonosov}.
In May 2023, a preprint of Per H. Enflo appeared on arXiv which, if correct,
shows that every bounded linear operator
on a Hilbert space has a closed nontrivial invariant subspace \cite{Enflo}.
In July 2023, a second and independent preprint of Charles W. Neville appeared on arXiv
claiming the solution of the problem for separable Hilbert spaces \cite{Neville}.
Invariant subspaces of adjointable maps in a Hilbert C*-module, approached through the lens of Banach lattices,
have been investigated by Katsoulis \cite{Katso}.

In this paper we are going to study invariant submodules of Hilbert C*-modules for bounded adjointable operators on Hilbert C*-modules
and prove some basic facts and properties for modular operators which reduce Hilbertian submodules.
Some fundamental questions concerning an adjointable operator $T$ can be translated to questions about invariant submodules of $T$.
We state the representation of the set of solutions to the operator equation $XTX= TX$
on Hilbert C*-modules by taking advantage of invariant submodules of $T$.
Let $T$ be a generalized invertible modular operator and $W$ a
complemented reducing submodule. If zero is not in the numerical range
of $T$, then $W$ is a reducing submodule under the generalized inverse $T^{ \dagger }$ of $T$.
This result is also remarkable  in the context of Hilbert spaces.
In the two last sections of our paper, we
consider the C*-algebras of
compact operators and finite dimensional C*-algebras and simplify our results for Hilbert modules over those C*-algebras of coefficients.
We obtain a version of Lomonosov's theorem for compact operators on (not necessarily finitely generated) Hilbert C*-modules.
The invariant submodules of a given bounded adjointable operator shed light on its structure. Indeed,
our results are useful to characterize
some representation and decompositions of module maps in C*-Fredholm theory and
theoretical physics \cite{PY}, which will be studied elsewhere.

Throughout the present paper we assume $ \mathcal{A}$ to be an
arbitrary C*-algebra.  We use
the notations $Ker( \cdot )$ and $Ran( \cdot )$ for kernel
and range of operators, respectively.

The notion of a Hilbert C*-module is a generalization of the notion
of a Hilbert space. However, some familiar properties
of Hilbert spaces like Pythagoras' equality, self-duality, and
even decomposition into orthogonal complements do not hold in the
framework Hilbert modules. The first use of such modules was made by
Kaplansky \cite{op5} and then studied systematically in the paper of
Paschke \cite{op8}. Let us shortly review what a Hilbert C*-module is.

Suppose that $ \mathcal{A} $ is an
arbitrary C*-algebra and $E$ is a linear space
which is a right $\mathcal{A}$-module and the scalar multiplication
satisfies $ \lambda (xa)=x(\lambda a)=(\lambda x)a $ for all $ x\in E, ~a\in
\mathcal{A}, \lambda \in \mathbb{C}$. The $\mathcal{A}$-module $E$ is called a
pre-Hilbert $ \mathcal{A} $-module if there exists an $\mathcal{A}$-valued
map $ \langle .,.\rangle: E \times E \to  \mathcal{A} $ with the
following properties:
\begin{enumerate}
\item $ \langle x, y+\lambda z\rangle =\langle x,y\rangle +\lambda \langle
 x,z\rangle $;  for all $ x,y,z\in E ,\lambda \in \mathbb{C},$
\item $ \langle x,ya\rangle=\langle x,y \rangle a;$  for all $x,y\in E $ and  $ a\in \mathcal{A}$,
\item $ \langle x,y\rangle ^{\ast}=\langle y,x \rangle;$ for all $x,y\in E$,
\item $ \langle x,x\rangle \geq 0  $ and $ \langle x,x\rangle=0 $ if and only if $ x=0.$
\end{enumerate}
The $\mathcal{A}$-module  $E$ is called a Hilbert C*-module
if $E$ is complete with respect to the norm $ \Vert x\Vert =\Vert \langle
 x,x\rangle \Vert ^{1/2}.$ For any pair of Hilbert C*-modules $ E_{1}$ and
$ E_{2}$, we  define
$ E_{1}\oplus E_{2}= \{  (e_1,e_2) | ~ e_{1}\in E_{1} ~{\rm and} ~e_{2}\in E_{2} \}$
which is also a Hilbert C*-module whose $ \mathcal{A} $-valued inner product
is given by $$ \langle (x_1,y_1), (x_2,y_2) \rangle =\langle x_{1},x_{2} \rangle +\langle y_{1},y_{2}
 \rangle, ~~{\rm  for}~ x_1, x_2 \in E_1~ {\rm and} ~y_1, y_2 \in E_2.$$

If $W$ is a (possibly non-closed) $\mathcal{A}$-submodule
of $E$, then $W^\bot :=\{ y \in E: ~ \langle x,y \rangle=0,~ \ {\rm for} \ {\rm
all}\ x \in W \} $ is a closed $\mathcal A$-submodule of $E$ and
$ \overline{W} \subseteq W^{ \perp \, \perp}$. A Hilbert $\mathcal{A}$-submodule $W$ of a Hilbert $\mathcal{A}$-module $E$ is
orthogonally complemented if $W$ and its orthogonal complement
$W^\bot$ yield $E=W \oplus W^\bot $, in this case, $W$ and its
biorthogonal complement $W^{ \perp \, \perp}$ coincide. By a \emph{complemented submodule} of $E$ we
mean one that is orthogonally complemented. Note that every Hilbert space is a
Hilbert $ \mathbb{C}$-module and every C*-algebra $\mathcal{A}$
can be regarded as a Hilbert $\mathcal{A}$-module via
$\langle a, b\rangle = a^*b$ when $a, b \in \mathcal{A}$.
We denote by
$\mathcal{L}(E,F)$ the Banach space of all adjointable operators from
$E$ to $F$, i.e., all $\mathcal{A}$-linear maps $T :E \rightarrow F$ such
that there exists $T^*:F \rightarrow E$ with the property $
\langle Tx,y \rangle =\langle x,T^*y \rangle$ for all $ x \in
E$ and $y \in F$. We set $\mathcal{L}(E):=\mathcal{L}(E,E)$. A bounded adjointable operator $V \in
\mathcal{L}(E)$ is called a partial isometry if $VV^*V=V$, see
\cite{SHA/PARTIAL} for some equivalent conditions. For the basic
theory of Hilbert C*-modules we refer to the books \cite{LAN, MAN3,
WEG} and the papers \cite{B-G, FR3, FR2, FR1}.

Closed submodules of Hilbert modules need not to be orthogonally
complemented at all, however we have the following well known
results. Suppose $T$ in $ \mathcal{L}(E,F)$, the operator $T$ has
closed range if and only if $T^*$ has. In this case, $E=Ker(T)
\oplus Ran(T^*)$ and $F=Ker(T^*) \oplus Ran(T)$, cf.
\cite[Theorem 3.2]{LAN}. In view of \cite[Lemma
2.1]{SHA/PARTIAL}, $Ran(T)$ is closed if and only if $Ran(T
\,T^*)$ is, and in this case, $Ran(T)=Ran(T\, T^*)$.

\begin{remark} \label{projection1}
For every complemented submodule $W$ of any Hilbert C*-module $E$ there
exists a unique orthogonal projection $P_W$ in $\mathcal{L}(E)$ such that $Ran(P_W)=W$ and
hence $Ker(P_W)=W^{ \perp}$.
\end{remark}

\section{Invariant submodules}
We first define invariant submodules for bounded adjointable operators on Hilbert C*-modules
and prove some properties for modular operators which reduce Hilbertian submodules. Our results shed light on the structure
and decomposition of modular operators on Hilbertian modules.

\begin{definition} \label{invariansub} Suppose $T \in \mathcal{L}(E)$. A closed submodule $W$
of a Hilbert $\mathcal{A}$-module $E$ is called $T$-invariant if  $T(W)$ is contained in $W$.
\end{definition}

\begin{lemma}\label{12} Let $T, S \in \mathcal{L}(E)$.
\begin{enumerate}
  \item If $TS=ST$, then $Ker(T)$ and $ \overline{Ran(T)}$ are $S$-invariant.
In particular, $Ker( p(T))$ and $ \overline{Ran(p(T))}$
are $S$-invariant for every polynomial $p$.
  \item If $E$ has at least two generators, and $T$ has no nontrivial invariant
submodule, then $Ker(T)=\{ 0 \}$ and $ \overline{Ran(T)} = E$.
\end{enumerate}
\end{lemma}

\begin{lemma} \label{14}
Let $T, S \in \mathcal{L}(E)$.
\begin{enumerate}
  \item If $T$ and $S$ are nonzero operators and $TS=0$, then $Ker(T)$ and $ \overline{ Ran(S)}$
  are nontrivial invariant submodules for both $T$ and $S$.
  \item If the operator equation $STS = TS$ has a nontrivial solution
  (i.e., a solution $S \in \mathcal{L}(E)$ such that $0 \neq S \neq I$), then $T$ has a nontrivial invariant submodule.
\end{enumerate}

\end{lemma}
\begin{proof}
Let $TS=0$, then  $Ran(S)  \subseteq Ker(T)$ and hence $S(Ker(T)) \subseteq Ran(S) \subseteq Ker(T)$. If $S \neq 0$, then
$ Ran(S) \neq \{ 0 \}$ so that $Ker(T) \neq \{ 0 \}$. If $T \neq 0$,  then $Ker(T) \neq E$ and so
$ \overline{ Ran(S)} \neq E$ because $Ker(T)$ is closed. Therefore $\{ 0 \} \neq Ker(T) \neq E$ and
$\{ 0 \} \neq \overline{Ran(S)} \neq E$, and $Ker(T)$ and $ \overline{Ran(S)}$ are nontrivial invariant submodules for $S$ and $T$, respectively.
According to Part (1) in Lemma \ref{12} and the fact that every operator commutes with itself,
the submodules $Ker(T)$ and $ \overline{ Ran(S)}$ are always invariant for both $T$ and $S$.

To prove the second part, suppose $TS = 0$. Then $T$ has a nontrivial invariant submodule according to the first part.
Suppose $TS \neq 0$.  Since $(I-S)TS=0$ and $S\neq I$, we get
$$\{ 0 \} \neq Ran(TS) \subseteq Ker(I-S) \neq E,$$ and so $Ker(I-S)$ is nontrivial. If $x\in Ker(I-S)$, then $Sx=x$ and so
$$Tx=TSx=STSx=STx,$$ which implies $Tx \in Ker(I-S)$. Therefore, the nontrivial submodule $Ker(I-S)$ is $T$-invariant.
\end{proof}

In the Hilbert modules context, one needs to add the extra condition, orthogonally complementing of closed submodules,
in order to get a reasonably good theory. This ensures the existing of orthogonal projection onto a closed submodule. There are
several examples of orthogonally complementing submodules, namely,
every norm closed submodule of every Hilbert C*-module over an arbitrary C*-algebra of compact operators, is automatically
a complemented submodule \cite{B-G, SCH}.

\begin{theorem} \label{41}
Let $T \in \mathcal{L}(E)$. Then the following assertions are pairwise equivalent.
\begin{enumerate}
  \item $T$ has a nontrivial complemented invariant submodule.
  \item There exists a nontrivial projection $P \in \mathcal{L}(E)$ such that $PTP = TP$.
  \item The operator equation $STS = TS$ has a nontrivial solution (i.e., a
solution $S \in \mathcal{L}(E)$ such that $0 \neq S \neq I$) such that $ \overline{ Ran(S)}$ and $Ker(I-S)$ are
complemented submodules in $E$.
\end{enumerate}
\end{theorem}

\begin{proof}
Let $W$ be a nontrivial  complemented invariant submodule of $E$, consider the orthogonal projection $P=P_W \in \mathcal{L}(E)$ onto $W$.
For any $x \in E$, $P x \in W$ and by invariance $T(P x) \in W$, whence
$$ P TP x=P(TP x) = TP x,$$ from which it immediately follows that $PTP=TP$.
Conversely, if the the latest identity holds, then $W=Ran(P)$ is $T$-invariant. Since $0 \neq  Ran(P) \neq E$ if and only if
$0 \neq  P \neq I$, it follows that $W=Ran(P)$ is nontrivial if and only if $P$ is.
Hence (1) is equivalent to (2). To conclude the proof, (2) trivially implies (3),
while (3) implies (2) by Lemma \ref{14} upon taking $P=P_W$, where $W=Ker(I-S)$.

\end{proof}

\begin{corollary} \label{RadjaviRosenthal}
Let $T, S \in \mathcal{L}(E)$ and $STS = TS$ such that $0 \neq S \neq I$ and let $ \overline{ Ran(S)}$ and $Ker(I-S)$ be
complemented submodules in $E$ (i.e., one of the conditions
in Theorem \ref{41} is fulfilled). Then $T$ maps the kernel of $(I-S)^n$ into to the kernel of $I-S$ for each positive integer $n$.
Consequently, $Ker((I-S)^n)$ is $T$-invariant for each $n$.
\end{corollary}
\begin{proof}
Suppose $STS = TS$ and $x \in Ker((I-S)^n)$ each positive integer $n$. According to the binomial theorem, we get $x=Sp(S)x$ for some
polynomial $p$. We therefore have $$(I-S)Tx=(I-S)TSp(S)x=0,$$ that is $Tx \in Ker(I-S)$.

\end{proof}

The majorization and range inclusion of operators on Hilbert C*-modules are investigated in \cite{FangMoXu, Xu/Sheng, Zhang}.
Let $T, S \in \mathcal{L}(E)$ and let $ Ran(T)$ be a closed submodule in $X$. There exists a solution $X$
of the so-called Douglas equation $TX=S$ if and only if $Ran(S) \subseteq Ran(T)$. In this case, the set of all solutions
of the equation $TX=S$ is given by
\begin{equation} \label{solu}
\mathcal{S}= \{ T^{ \dag}S+(1-T^{ \dag}T)Z \, :~ \mathrm{~for~an ~arbitrary ~operator} ~ Z \in \mathcal{L}(E) \}.
\end{equation} \label{law0}
Here the Moore-Penrose inverse $T^{ \dagger }$ of $T$ is an element $X \in \mathcal{L}(E)$ which satisfies
$$ TXT=T, ~XTX=X, ~(TX)^*=TX, ~ \mathrm{and}~~(XT)^*=XT.$$
The (bounded) Moore-Penrose inverse of a
bounded adjointable operator $T$ exists if and only if $T$ has a closed range \cite{LLX, Xu/Sheng, Zhang}.
If the Moore-Penrose inverse $T^{ \dag} \in \mathcal{L}(E)$ exists, it unique and satisfies $T^{ \dag}T= P_{Ran(T^*)}$ and $TT^{ \dag}= P_{Ran(T)}$.

One should be aware that
a bounded adjointable operator may have an unbounded Moore-Penrose inverse.
Indeed, the (possibly unbounded) Moore-Penrose inverse of a
bounded adjointable operator $T$ exists if and only if the closures
of the ranges of $T$ and $T^*$ are (orthogonally) complemented \cite{FS1, FS2, GUL}.

Solvability and different solutions of the operator equation $XAX = AX$ on Hilbert spaces
have been studied first by Holbrook, Nordren, Radjavi and Rosenthal \cite{HNRR}, and recently by
Cvetkovi\'c-Ili\'c \cite{CVI}. We reformulate
solvability and solutions of this equation in the framework of Hilbert C*-modules as follows.

\begin{corollary} \label{Wong}
Let $T, S \in \mathcal{L}(E)$ and $STS = TS$ such that $0 \neq S \neq I$ and let $ \overline{ Ran(S)}$ and $Ker(I-S)$ be
complemented submodules in $E$ (i.e., one of the conditions
in Theorem \ref{41} is fulfilled). If  $W$ is a complemented $T$-invariant submodule in $E$ and  $Ran(TP_{W})$ is a closed submodule in $E$,
then the set of all solutions of the equation $XTX = TX$ is given by
$$ \mathcal{S}= \{ P_{Ran(TP_{W})}+P_{W}Z(I-P_{Ran(TP_{W})}) \, :~ \mathrm{~for~an ~arbitrary ~operator} ~ Z \in \mathcal{L}(E) \}.$$
\end{corollary}

\begin{proof} If there exists $S$ such that $STS = TS$, then
$$STP_{W}=TP_{W}~ \mathrm{and}~ (I-P_{W})S=0$$
hold for $W= \overline{ Ran(S)}$. In view of the equality $(P_{W}T^*)S^*=P_{W}T^*$ and (\ref{solu}), the set of all solutions is given by
\begin{eqnarray*}
 \mathcal{S} &=&
 \{ ((P_{W}T^*)^{ \dag}  (P_{W}T^*))^*+ P_{W}Z(I-P_{Ran(TP_{W})}) \, :~ \mathrm{~for~an ~arbitrary ~operator} ~ Z \in \mathcal{L}(E) \} \\
             &=&
 \{ P_{Ran(TP_{W})}+P_{W}Z(I-P_{Ran(TP_{W})}) \, :~ \mathrm{~for~an ~arbitrary ~operator} ~ Z \in \mathcal{L}(E) \}.
\end{eqnarray*}
\end{proof}

\begin{proposition} \label{42}
Let $T \in \mathcal{L}(E)$ and let $W$ be a
complemented submodule of $E$. If $W$ is $T$-invariant, then $T$ can be written with
respect to the decomposition $E=W \oplus W^{ \perp}$ as
\begin{equation*} \label{law0}
T=\left[\begin{array}{cc}
T_{|_{W}} & B \\
0 & D \\
\end{array}\right]: \left[\begin{array}{c}
W \\
W^{ \perp} \\
\end{array}\right] \to\left[\begin{array}{c}
W \\
W^{ \perp} \\
\end{array}\right].
\end{equation*}
Conversely, take $A \in \mathcal{L}(W)$, $B \in \mathcal{L}(W^{ \perp}, W)$ and  $D \in \mathcal{L}(W^{ \perp})$.
If
\begin{equation*} \label{law1}
T=\left[\begin{array}{cc}
A & B \\
0 & D \\
\end{array}\right]: W \oplus W^{ \perp}=X \to W \oplus W^{ \perp}=X,
\end{equation*}
then $W$ is $T$-invariant and $A= T_{|_{W}} $.
\end{proposition}

\begin{proof}
Suppose $T$ possesses the matrix decomposition $ \left[ \begin{smallmatrix} A & B \\
C & D   \end{smallmatrix} \right]$ with respect to the decomposition $X=W \oplus W^{ \perp}$, with
$A \in \mathcal{L}(W)$, $B \in \mathcal{L}(W^{ \perp}, W)$, $C \in \mathcal{L}(W, W^{ \perp})$ and
$D \in \mathcal{L}(W^{ \perp})$. If $W$ is $T$-invariant, then $C=0$ because $W \cap W^{ \perp} = \{ 0 \}$.

Now suppose $v \in W$ and $C=0$, then $Tv = \left[ \begin{smallmatrix} A & B \\
0 & D   \end{smallmatrix} \right] \, \left[ \begin{smallmatrix} v  \\
0  \end{smallmatrix} \right]= \left[ \begin{smallmatrix} Av  \\
0  \end{smallmatrix} \right] \in W \oplus \{ 0 \}$. Hence $W$ is $T$-invariant and $A= T_{|_{W}} $.
\end{proof}

\begin{definition} \label{reducedsub} Suppose $T \in \mathcal{L}(E)$. A closed submodule $W$
of a Hilbert $\mathcal{A}$-module $E$ is called a reducing submodule for $T$ if $W$ and its orthogonal complement $W^{ \perp}$
are both $T$-invariant.
\end{definition}

We say that $W$ \textit{reduces} $T \in \mathcal{L}(E)$ when $W$
is a reducing submodule for $T$. If $W$ reduces $T$ and $\{ 0 \} \neq W \neq E$, then $W$
is a nontrivial reducing submodule for $T$.
It is easy to see that a submodule $W$ of a Hilbert C*-module $E$ reduces $T \in \mathcal{L}(E)$ if
and only if $W$ is invariant for both $T$ and $T^*$.
An operator $T$ acting on a Hilbert C*-module $E$, with at least two generators,
is \textit{reducible} if it has a nontrivial reducing submodule.

\begin{proposition} \label{4346}
Let $T$ be an element in the C*-algebra $\mathcal{L}(E)$. Then
\begin{enumerate}
\item  the complemented submodule $W$ is invariant for every operator that commutes with $T$ if and only
if $W^{ \perp}$ is invariant for every operator that commutes with $T^*$;
\item if $T$ commutes with an orthogonal projection $P$, then $Ran(P)$ is a reducing submodule for $T$;
\item $T$ has a nontrivial, complemented, reducing submodule if and only if $T$ commutes with a nontrivial orthogonal
projection.
\end{enumerate}
\end{proposition}

\begin{proof}
Let $comm\{T \}$ be the commutant of $T$ in $ \mathcal{L}(E)$.
It is clear that $S \in comm\{T \}$ if and only if $S^* \in comm\{T^* \}$.
Suppose the complemented submodule $W$ is invariant for every operator $S \in \mathcal{L}(E)$ that commutes with $T$.
Take an arbitrary $y \in W^{ \perp}$, since $Sx \in W$ whenever $x \in W$, then
$$\langle x, S^*y \rangle = \langle Sx, y \rangle=0, ~ \mathrm{for~every~} x\in W,$$
which means $S^*(W^{ \perp} ) \subseteq W^{ \perp}$. Dually, if $W^{ \perp}$ is invariant
for every operator that commutes with $T^*$, then $W^{ \perp \, \perp}=W$
is invariant for every operator that commutes with $T^{**} = T$. This proves $(1)$.

If $PT=TP$, then $PT^*=T^*P$, and hence $Ran(P)$ is $T$-invariant and $T^*$-invariant, that
is $(2)$ holds.

To prove $(3)$, suppose $P$ is a nontrivial orthogonal projection on $E$ and
$P \in comm\{ T \}$, then
$Ran(P)$ is a nontrivial reducing subspace for $T$; that is, $T$ is reducible.
Conversely, suppose $T$ is reducible so that there exists a nontrivial
submodule $W$ such that both $W$ and $W^{ \perp }$ are $T$-invariant. Since $W$
is $T$-invariant, the second part of Theorem \ref{41} implies that the nontrivial orthogonal projection $P$
onto $W$ is such that $PTP = TP$. Similarly, since
$W^{ \perp }$ is $T$-invariant, it also follows that the complementary projection
$I -P$ onto $W^{ \perp }$ is such that $(I-P)T(I-P) = T(I-P)$, and hence $PTP = PT$, consequently,
$PT = TP$.
\end{proof}

\begin{remark} \label{47}
Let $W$ be a complemented submodule of a Hilbert C*-module $E$. Take any $T$
in the C*-algebra $\mathcal{L}(E)$ and consider the orthogonal projection $P$ on $E$ with $Ran(P)=W$. Then the
following assertions on the restrictions of $T$ and $T^*$ to $W$ hold.
\begin{enumerate}
  \item If $W$ is $T$-invariant, then $(T_{|_{W}})^* = PT^{*}_{ ~ |_{W}}$.
  \item If $W$ reduces $T$, then $(T_{|_{W}})^*=T^{*}_{ ~ |_{W}}$.
\end{enumerate}
\end{remark}

\begin{proposition} \label{48}
Let $T \in \mathcal{L}(E)$ and let $W$ be a
complemented submodule of $E$. If $W$ and $W^{ \perp}$ are $T$-invariant (i.e., $W$ is reducing submodule for $T$), then $T$ can be written with
respect to the decomposition $E=W \oplus W^{ \perp}$ as
\begin{equation*} \label{law0}
T=\left[\begin{array}{cc}
T_{|_{W}} & 0 \\
0 & T_{|_{W^{ \perp}}} \\
\end{array}\right]: \left[\begin{array}{c}
W \\
W^{ \perp} \\
\end{array}\right] \to\left[\begin{array}{c}
W \\
W^{ \perp} \\
\end{array}\right].
\end{equation*}
Conversely, take $A \in \mathcal{L}(W)$ and  $D \in \mathcal{L}(W^{ \perp})$.
If
\begin{equation*} \label{law1}
T=\left[\begin{array}{cc}
A & 0 \\
0 & D \\
\end{array}\right]: W \oplus W^{ \perp}=E \to W \oplus W^{ \perp}=E,
\end{equation*}
then $W$ reduces $T$ and $A= T_{|_{W}} $ and $D= T_{|_{W^{ \perp}}} $.
\end{proposition}

\begin{proof} Both submodules $W$ and $W^{ \perp}$ are $T$-invariant.
The proof of this proposition follows the same steps as the proof of Proposition \ref{42}.
\end{proof}

In what follows, we consider the C*-numerical range for an operator on a Hilbert C*-module $E$. This definition
differs from the standard numerical range used for the unital Banach algebra $\mathcal{L}(E)$.
\begin{theorem} \label{ginverse}
Let $T \in \mathcal{L}(E)$ be Moore-Penrose invertible and $W$ a
complemented submodule of $E$. If $W$ is a reducing submodule for $T$
and
$$0\notin \omega(T):= \{  \langle Tx,x \rangle  : ~x\in E, ~ \| x \| = 1 \},$$
the C*-numerical range
of $T$, then
$W$ is a reducing submodule under the Moore-Penrose inverse $T^{ \dagger }$ of $T$.
\end{theorem}

\begin{proof}
According to the Moore-Penrose invertibility of $T$ and Proposition \ref{48}, the operator
\begin{equation*} \label{law0}
T=\left[\begin{array}{cc}
T_{|_{W}} & 0 \\
0 & T_{|_{W^{ \perp}}} \\
\end{array}\right]: \left[\begin{array}{c}
W \\
W^{ \perp} \\
\end{array}\right] \to\left[\begin{array}{c}
W \\
W^{ \perp} \\
\end{array}\right]
\end{equation*}
has a closed range. It follows that $T(W)$ is a closed submodule of $W$ and $W=T(W) \oplus T(W)^{ \perp}$, cf.
\cite[Theorem 3.2]{LAN}. We claim that $W \cap T(W)^{ \perp}=\{ 0 \}$. To see this, let $x$ be a unit vector in $W \cap T(W)^{ \perp}$. The
inclusion $W \cap T(W)^{ \perp} \subset W$ implies $Tx \in T(W)$ and hence $ \langle Tx, x \rangle =0$. Since $0\notin \omega(T)$, $ W \cap T(W)^{ \perp}$
must be zero, which implies $T(W)=W$. The inclusion $T^*(W) \subset W$ together with the equalities $T(W)=W$ and $Ran(T^{ \dagger }T)=Ran(T^*)$ imply that
$$T^{ \dagger }(W) = T^{ \dagger }T(W) \subset W.$$ Similarly, the submodule $W^{ \perp}$ is an invariant submodule under
the Moore-Penrose inverse $T^{ \dagger }$ of $T$.
\end{proof}

\begin{corollary} \label{ginverse1}
Let $T \in \mathcal{L}(E)$ be invertible and $W$ a
complemented submodule of $E$. If $W$ is a reducing submodule for $T$
and
$$0\notin \omega(T):= \{  \langle Tx,x \rangle  :  ~x\in E, ~ \| x \| = 1 \},$$
the C*-numerical range
of $T$, then
$W$ is a reducing submodule under the inverse operator $T^{ -1}$.
\end{corollary}
The two above results are interesting even in the case of Hilbert spaces.
Two Hilbert C*-modules are said to be unitarily equivalent if there is a unitary transformation mapping between them \cite{LAN}.
Two operators $T \in \mathcal{L}(E)$ and $S \in \mathcal{L}(F)$ are \textit{unitarily equivalent}
if there is a unitary operator $U : F \to E$ such that $S = U^*TU$.

\begin{proposition} \label{49}
Let $T \in \mathcal{L}(E)$ and $S \in \mathcal{L}(F)$ be unitarily equivalent.
If $T$ has a nontrivial complemented reducing submodule, then $S$ has a nontrivial complemented reducing submodule, too.
That is, reducibility is preserved under unitary equivalence.
\end{proposition}

\begin{proof}
Take a unitary operator $U : F \to E$ such that $S = U^*TU$, in view of part (3) of Proposition \ref{4346},
there exists a nontrivial orthogonal projection $P \in \mathcal{L}(E)$ such that $TP=PT$. Put $Q=U^*PU$ then
$Q$ is a nontrivial orthogonal projection on $F$ which satisfies $QS-SQ = U^*(PT-TP)U=0$. Therefore, according
to Proposition \ref{4346}, the operator $S$ has a complemented reducing submodule.

\end{proof}

\begin{proposition} \label{410}
Let $T$ be a contraction in the C*-algebra $\mathcal{L}(E)$ and $W$ a nonzero complemented submodule of $E$.
If $W$ is $T$-invariant and $T_{|_{W}}$ is unitary, then $W$ reduces $T$.
\end{proposition}

\begin{proof}
Let  $U=T_{|_{W}}$, in view of the decomposition $E=W \oplus W^{ \perp}$  and Proposition \ref{42} we have
\begin{equation*} \label{law410}
T=\left[\begin{array}{cc}
U & B \\
0 & D \\
\end{array}\right] ~~~ \ ~~~ \mathrm{and} ~~~ \ ~~~
T^*T=\left[\begin{array}{cc}
I & U^*B \\
B^*U & B^*B+D^*D \\
\end{array}\right].
\end{equation*}
In view of \cite[Proposition 1.2]{LAN} and $\| T \|^2= \| T^* T \| \leq 1$, we get
\begin{eqnarray*}
\langle w , w \rangle + \langle B^*Uw , B^*Uw \rangle =  \langle \left[ \begin{smallmatrix} w  \\
B^*Uw  \end{smallmatrix} \right] , \left[ \begin{smallmatrix} w  \\
B^*Uw  \end{smallmatrix} \right] \rangle &=&
\langle T^*T(w,0) , T^*T(w,0) \rangle \\
& \leq &
\| T^*T \|^2 \,   \langle (w,0) , (w,0) \rangle \\
& \leq &
\langle w , w \rangle + \langle 0 , 0 \rangle.
\end{eqnarray*}
This implies that $B^*U=0$, and hence $B=B^* = 0$. Therefore, according
to Proposition \ref{48}, the submodule  $W$ reduces the operator
$T= \left[ \begin{smallmatrix} U & B \\ 0 & D   \end{smallmatrix} \right] =
\left[ \begin{smallmatrix} U & 0 \\ 0 & D   \end{smallmatrix} \right]$.
\end{proof}

We know that if every subspace of a Hilbert space is invariant under a linear
transformation $T$, then $T$ is a scalar transformation.
This fact motivates us to close this section with the following problem. We will
study such C*-algebras and Hilbert C*-modules elsewhere.
\begin{problem} \label{problem1}
Let $E$ be a Hilbert $ \mathcal{A}$-$\mathcal{A}$-bimodule. For which C*-algebras
$ \mathcal{A}$ does the following assertion hold?

If every closed submodule of $E$ is invariant under a bounded adjointable operator $T:E \to E$,
then $T$ is a C*-coefficient morphism, i.e.,
there exists a nonzero element $a \in \mathcal{A}$ such that $T=I. \,a$, where $I$ is the identity operator on $E$.

\end{problem}

\section{C*-algebras of compact operators}

Suppose that $\mathcal{A}$ is an arbitrary C*-algebra of
compact operators, i.e., $\mathcal{A}$ is a C*-subalgebra of $\mathcal{K}(H)$ for some Hilbert space $H$.
It is well-known that $\mathcal A$ has to be
of the form of a $c_{0}$-direct sum of elementary
C*-algebras $\mathcal{K}(H_{i})$ of all compact operators
acting on Hilbert spaces $H_{i}, \ i \in I$ \cite[Theorem 1.4.5]{ARV}.

Magajna and Schweizer  have demonstrated, respectively, that
C*-algebras of compact operators can be characterized by the
property that every norm closed (coinciding with its biorthogonal
complement, respectively) submodule of every Hilbert C*-module
over them is naturally an orthogonal summand \cite{FR1, SCH}.
Indeed,  C*-algebras of compact
operators have a special importance in Hilbert C*-modules theory.
More general characteristics of the category
of Hilbert C*-modules over C*-algebras,
which characterize
precisely the C*-algebras of compact operators have been found
by the author and Frank in \cite{FR1, FS1, FS2, SHA/PRODUCT, SHA/inverselimits} and by Baki\'c and Gulja\v{s} in \cite{B-G, GUL}.

When we deal with Hilbert modules over C*-algebra of
compact operators, we can shorten the results of the previous section, namely,  Theorem \ref{41} and Corollary \ref{Wong} read as follows.

\begin{corollary} \label{411}
Let $E$ be a Hilbert C*-module over an arbitrary C*-algebra of compact operators and let  $T \in \mathcal{L}(E)$.
Then the following assertions are pairwise equivalent.
\begin{enumerate}
  \item $T$ has a nontrivial invariant submodule.
  \item There exists a nontrivial projection $P \in \mathcal{L}(E)$ such that $PTP = TP$.
  \item The operator equation $STS = TS$ has a nontrivial solution (i.e., a
solution $S \in \mathcal{L}(E)$ such that $0 \neq S \neq I$).
\end{enumerate}
\end{corollary}

\begin{corollary} \label{Wong1}
Let $E$ be a Hilbert C*-module over an arbitrary C*-algebra of compact operators and let $T, S \in \mathcal{L}(E)$
such that $0 \neq S \neq I$. If  $W$ is a $T$-invariant submodule in $E$ and  $Ran(TP_{W})$ is a closed submodule in $E$,
then the set of all solutions
of the equation $STS = TS$ is given by
$$ \mathcal{S}= \{ P_{Ran(TP_{W})}+P_{W}Z(I-P_{Ran(TP_{W})}) \, :~ \mathrm{~for~an ~arbitrary ~operator} ~ Z \in \mathcal{L}(E) \}.$$
\end{corollary}

\section{Lomonosov's theorem}

The author's interest in Lomonosov's theorem was stimulated by the old observation of von Neumann,
Aronszajn, and Smith \cite{Aronszajn} regarding the existence of proper invariant subspaces for completely
continuous operators in a Hilbert space.

In Banach spaces theory, compact operators are linear operators on Banach spaces that map
bounded sets to relatively compact sets.
In the case of a Hilbert C*-modules, the compact
operators are the closure of the finite rank operators in the uniform operator topology.
In this theory, a rank one operator  $\theta _{x,y}
:E \rightarrow E$ is defined by $\theta_{x,y}(z) = x \langle y,z \rangle $,  for given elements $x,y\in E$.
Then each $\theta_{x,y}$ is a bounded adjointable operator on $E$, with the adjoint
$(\theta_{x,y} )^*=\theta_{y,x}$. The closure of the span of $\{
\theta_{x,y}: x,y \in E \}$ in $ \mathcal{L}(E) $ is denoted by $ \mathcal{K}(E)$, and
elements from this set will be called ``\textit{compact}" operators. We may sometimes use the notation
$ \mathcal{K}_{ \mathcal{A}}(E)$ in place of $ \mathcal{K}(E)$, to make explicit the underlying C*-algebra.
In general, operators on Hilbert modules feature properties that do not appear
in the Hilbert (or Banach) spaces case, cf. \cite{FR2, LAN, MAN2, MAN3}.
The compact operators on Hilbert modules over a finite dimensional C*-algebra are notable
in that
they share as much similarity with matrices as one can expect from a general operator.
Frank has shown that every
Hilbert C*-module over an arbitrary finite dimensional C*-algebra of coefficients is a
self-dual Hilbert module, cf. \cite[Proposition 4.4]{FR3}.

Lomonosov's theorem states that every nonzero compact operator on a complex
Hilbert space has a proper
nonzero hyperinvariant subspace.
In this section, we are going to prove Lomonosov's theorem for compact modular operators on
Hilbert modules over finite dimensional C*-algebras, to do so, we
first need the following results.


\begin{lemma} \label{Arambasic1} ( \cite[Theorem 2.3]{Arambasic} $\&$ \cite[Proposition 2.1]{CIMS})
Let $ \mathcal{A}$ be a finite dimensional C*-algebra and $E$ a right Hilbert $\mathcal{A}$-module. For every bounded sequence
$(\zeta_n)$ in $E$ there are a subsequence $(\zeta_{n_k})$ of $(\zeta_n)$ and $\zeta \in E$ such that
\begin{equation}\label{Aram1}
\langle v , \zeta_{n_k} \rangle \to \langle v , \zeta \rangle, ~~~~ \mathrm{for~every~} v \in E,
\end{equation}
and
\begin{equation}\label{Aram2}
\| T \zeta_{n_k} - T \zeta\| \to 0, ~~~~ \mathrm{for~every~} T\in \mathcal{K}(E).
\end{equation}
\end{lemma}
\begin{proof}
Let $(\zeta_n)$ be a bounded sequence in $E$, then $\| \langle v , \zeta_{n} \rangle \| \leq M \|v \|$, for some
$M$ and all $v \in E$. Since $\mathcal{A}$ is a finite dimensional space, the $\mathcal{A}$-valued sequence $\langle v , \zeta_{n} \rangle$
has a convergent subsequence $\langle v , \zeta_{n_k} \rangle$ which is convergent to some $g(v)$ in $\mathcal{A}$ for any
$v \in E$. Then $g: E \to \mathcal{A}$, $g(v)= \lim \, \langle v , \zeta_{n_k} \rangle$ is an $ \mathcal{A}$-linear map which satisfies
$\| g(v) \| \leq M \|v \|$. Utilizing \cite[Proposition 4.4]{FR3}, self-duality of $E$ implies that there exists
$ \zeta \in E$ such that
$$ \langle v , \zeta_{n_k} \rangle \to g(v)=\langle v , \zeta \rangle, ~~~~ \mathrm{for~every~} v \in E.$$
To get (\ref{Aram2}), it remains to note that $\mathcal{K}(E)$ is closure of the span of $\{
\theta_{u,v}: u,v \in E \}$ in $ \mathcal{L}(E) $ and
$$ \| \theta_{u,v} (\zeta_{n_k}) - \theta_{u,v} (\zeta) \| = \| u \langle v , \zeta_{n_k} \rangle - u \langle v , \zeta \rangle \| \to 0,$$
for~every $u,v\in E$.
\end{proof}

The following lemma follows from Riesz's lemma for Banach spaces, which can be found in Chapter VII, \S7 of the book \cite{JBConway}.
\begin{lemma} \label{Riesz}
Let $E$ be a Hilbert C*-module over an arbitrary C*-algebra $ \mathcal{A}$, and let $F$
be a closed submodule of $E$. For all $0 < \epsilon < 1$, there exists
$x \in E$ such that $\| x \|=1$ and
$$ \| x - F \| \geq \epsilon,$$
where $$\| x - F \|=  \mathrm{ inf} \{ \| x - u \| \, : \ u\in F \}.$$
\end{lemma}

\begin{lemma} \label{Conwayp2150}
Let $E$ be a Hilbert C*-module over a finite dimensional C*-algebra $ \mathcal{ A}$ and $K \in \mathcal{K}(E)$.
Then $Ran(I-K)$ is an orthogonal summand.
\end{lemma}

\begin{proof}
Let $y$ be in the closure of $Ran(L)$ and $L=I-K$. Then there exists a sequence $(x_n)$ in $E$ such that $Lx_n=x_n - Kx_n$
converges to $y$. Using compactness of $K$ and Lemma \ref{Arambasic1}, there exists a sequence $(u_n)$ in the closed submodule $Ker(L)$ of $E$ such that
$$ \| x_n - u_n \|=  \mathrm{ inf} \{ \| x_n - u \| \, : \ u\in Ker(L) \} .$$
Set $ \zeta_n := x_n - u_n$. We claim $ \zeta_n$ is a bounded sequence in $E$, because in the contrary case there is a
subsequence $( \zeta_{n_k})$ such that $\| \zeta_{n_k} \| \geq k$. By compactness of $K$ and Lemma \ref{Arambasic1}, there exists
a subsequence $(v_{k_j})$ of $v_k := \frac{ \zeta_{n_k} }{ \| \zeta_{n_k} \|}$ and $v$ such that
$ Kv_{k_j} \to v$.  Since
$$ \| Lv_k \|= \frac{ \| L \zeta_{n_k} \| }{ \| \zeta_{n_k} \| } \leq \frac{ \| L \zeta_{n_k} \| }{ k } \to 0, ~~\mathrm{as}~ k \to \infty,$$
we have $v_{k_j}=Lv_{k_j} + K v_{k_j} \to v$, which implies $0= \lim L v_{k_j} = Lv$. Since
$ u_{n_k} + \| \zeta_{n_k} \|v$ is in $Ker(L)$ we have
$$ \| v_k - v \|= \frac{1}{ \| \zeta_{n_k} \| } \, \| x_{n_k} - (u_{n_k} + \| \zeta_{n_k} \|v) \| \geq \frac{1}{ \| \zeta_{n_k} \| } \,
\inf \| x_{n_k} - u \| = 1,$$
which is a contradiction with $ v_{k_j} \to v$, that is, the sequence $( \zeta_{n} )$ is bounded.
In view of Lemma \ref{Arambasic1}, the bounded sequence $( \zeta_{n})$
has a subsequence
$( \zeta_{n_k})$ such that $ K \zeta_{n_k}$ is convergent. Hence, $ \zeta_{n_k}= L \zeta_{n_k} + K \zeta_{n_k}$ converges
to an element
$ \zeta \in E$. On the other hand, $Lx_n \to y$ and $Lx_{n_k}-Lu_{n_k} = L \zeta_{n_k} \to L \zeta $, i.e., $y=L \zeta$.
Consequently, the closed submodule $Ran(L)$ is an orthogonal summand by \cite[Theorem 3.2]{LAN}.
\end{proof}

\begin{corollary} \label{lambda}
Let $E$ be a Hilbert C*-module over a finite dimensional C*-algebra $ \mathcal{ A}$ and $K \in \mathcal{K}(E)$.
Then $Ran( \lambda I -K)$ is an orthogonal summand for every nonzero $ \lambda \in \mathbb{C}$.
\end{corollary}

\begin{remark} \label{Wegge257}
Let $E$ be a Hilbert C*-module over a unital C*-algebra $ \mathcal{A}$. Then $E$ is
finitely generated if and only if $\mathcal{K}_{ \mathcal {A}}(E)$ is unital.
The proof of this fact follows from equivalence of (6) and (7) in
Remark 15.4.3 of Wegge-Olsen's book \cite{WEG}.
\end{remark}
The reader should be aware that every finite dimensional C*-algebra $ \mathcal{A}$ is unital. Indeed,
one can show that it is a direct sum of full matrix blocks which is a unital C*-algebra.

In the following lemma we summarize the corresponding results from the matrix case to
the spectral properties of compact operators on Hilbert C*-modules.

\begin{lemma} \label{Conwayp215}
Let $E$ be a Hilbert C*-module over a finite dimensional C*-algebra $ \mathcal{ A}$, $K \in \mathcal{K}_{ \mathcal {A}}(E)$ and
$$ \sigma (K) =  \{ \lambda \in \mathbb{C} :~ \lambda I -K {\rm ~is ~not ~invertible~in}~ \mathcal{L}_{ \mathcal {A}}(E) \}$$
the spectrum of $K$.
\begin{enumerate}
\item If $E$ is not a finitely generated Hilbert $ \mathcal{ A}$-module, then $0 \in \sigma (K)$.
\item Every nonzero $\lambda \in \sigma (K)$ is an eigenvalue of $K$, that is, $Ker( \lambda I -K ) \neq \{ 0 \}$.
\end{enumerate}
\end{lemma}
\begin{proof}
Let us assume $ 0  \notin \sigma (K)$, then there is a bounded adjointable operator $S$ on $E$ such that $SK=I$.
The operator $K$ is compact and so is $SK=I$, consequently $E$ is a finitely generated Hilbert $ \mathcal{ A}$-module by
Remark \ref{Wegge257}. This proves (1).

Without loss of generality, assume $1= \lambda \in \sigma (K)$ not being an eigenvalue, i.e. $L=I - K$
is injective but not surjective. Then $F_1=Ran(L)$ is a closed proper submodule of $E$ by Lemma \ref{Conwayp2150}. We set $F_n=Ran(L^n)$, since
$L=I-K$ is injective, $F_n$ is again, by Lemma \ref{Conwayp2150}, a closed proper submodule of $F_{n-1}$. We find the decreasing sequence of submodules
$$ F_1 \supset F_2 \supset \cdots \supset F_n \supset \cdots \supset F_m \cdots ,$$
in which all inclusions are proper. Utilizing Lemma \ref{Riesz}, we can choose unit vectors $f_n \in F_n$ such that
$\| f_n -F_{n+1} \| > 1/2$. Lemma \ref{Arambasic1} implies that $\{ Kf_n \}$ must contain a convergent subsequence, but for $n < m$,
$ Lf_n - Lf_m + f_m \in F_{n+1}$ and $$ \| Kf_n - Kf_m \| = \| Lf_n - Lf_m - (f_n - f_m) \| > 1/2, $$ which is a contradiction.
This proves (2).

\end{proof}


In an unpublished work, von Neumann proved that every compact operator on a complex
Hilbert space of dimension at least 2 has a proper nonzero invariant subspace. Aronszajn
and Smith generalized the result to Banach-space operators \cite{Aronszajn}.
In 1973, Lomonosov proved a sweeping generalization of this result  \cite{Lomonosov}. Not only
is his result much stronger but its proof is simpler than that of Aronszajn and Smith.
Traditionally, ``Lomonosov's theorem" refers to Corollary \ref{M292} below.


\begin{definition} \label{hyperinvariant}
An invariant submodule $W$ for $T \in \mathcal{L}(E)$ is hyperinvariant if $S(W)$ is contained in $W$
for every $S \in \mathcal{L}(E)$ with $TS=ST$.
\end{definition}

\begin{theorem} \label{M291}
Let $ \mathcal{A}$ be a finite dimensional C*-algebra and let $E$ be a
Hilbert $ \mathcal{A}$-module that is not finitely generated.
Then every nonzero compact operator on $E$ has a proper nonzero hyperinvariant submodule.
\end{theorem}

\begin{proof} Suppose $K:E \to E$ is a nonzero compact operator. Since the spectrum of $K$ is a nonempty compact subset of $ \mathbb{C}$,
we have to consider two possibilities: there exists a nonzero $ \lambda$ in $\sigma (K)$ or $\sigma (K) = \{0 \} $.

If there exists a nonzero $ \lambda \in \sigma (K)$, then $K$ has a nonzero eigenvalue $ \lambda$ by Lemma \ref{Conwayp215}.
Consequently, $W= Ker( \lambda I - K)$, the
eigenspace for $ \lambda$, is a nonzero $K$-invariant submodule.
Since $E$ is not finitely generated and $K$ is compact, $K \neq \lambda \, I$ and hence $W \neq E$.
If  $S \in \mathcal{L}(E)$ and $KS=SK$, then $K(Sx)= \lambda (Sx)$,
for each $x \in W$. In other words, $W$ is a $S$-invariant submodule,
which is indeed a proper nonzero hyperinvariant submodule for $K$.

We now suppose that $\sigma (K) = \{0 \} $. Then
\begin{equation}\label{M1}
\lim \| ( \gamma \, K)^n \|^{ \frac{1}{n} }= | \gamma | \lim \|  K^n \|^{ \frac{1}{n} } = 0,
~~~~~\mathrm{for~ all~} \gamma \in \mathbb{C}.
\end{equation}
Without loss of generality, suppose  $\| K \|=1$. Indeed, multiply by $ 1 / \|K \|$ which does not change
the invariant-submodule structure for $K$. Pick an $x_0 \in E$ such that $\| Kx_0 \| =1+ \delta $
for some positive $\delta$. Specifically, we have  $\| x_0 \| > 1$. Let $\mathfrak{B}$ denote the closed unit ball centered at $x_0$, i.e.,
$\mathfrak{B}= \{ x \in E: ~ \|x - x_0 \| \leq 1 \}$. Then $0$ is not in $\mathfrak{B}$ nor in the
closure of $K \mathfrak{B}$, since for $x \in \mathfrak{B}$,
$$ \| Kx \| \geq \|Kx_0 \| - \|K(x -x_0) \| > \delta > 0.$$
For each fixed $u \in E$, the set
$$ W_u= \{ Su: ~ S \mathrm{~ is~ a~ bounded~ adjointable~ operator~ with ~} SK=KS \}$$
a submodule that is invariant under any operator that commutes with $K$
(hence invariant under $K$ itself). Hence its closure is a hyperinvariant submodule for $K$,
it suffices to prove that there is a $u \in E$ such that $ \overline{W_u}$ is proper and nonzero.
If $ u \neq 0$ then $ W_u \neq \{ 0 \}$, so we are done unless $ W_u$ is dense in $E$ whenever
$ u \neq 0$.

We show that it is not possible that $ W_u$ be dense for all $u \neq 0$. Suppose the contrary; then for each
$u \neq 0$ there is an operator $S$ in $ \mathcal{L}(E)$  such that $KS=SK$ and $\| Su - x_0 \| < 1$. In the other words, if we let
$$ \mathcal{O}(S)= \{ u : ~ \| Su - x_0 \| < 1 \},$$
Then
$$ \cup \, \{ \mathcal{O}(S): ~ S \in \mathcal{L}(E), ~ SK=KS \}= E \, \backslash \{0 \}.$$
According to Lemma \ref{Arambasic1}, $  \overline{ K \, \mathfrak{B} }= K \, \mathfrak{B} $ is a compact subset of
$E \, \backslash \{0 \}$, which is covered by the open sets $\mathcal{O}(S)$ and hence, there are $S_1,S_2, \dots, S_n$
in $\mathcal{L}(E)$ that commute with $K$ such that $K \, \mathfrak{B}$ is contained in $ \bigcup\limits_{i=1}^{n} \, \mathcal{O}(S_i)$.
Since $x_0 \in \mathfrak{B}$ and  $Kx_0 \in K \, \mathfrak{B}$, we conclude that  $Kx_0 \in \mathcal{O}(S_{i_1})$ for some $i_1 \in \{1, \dots ,n \}$.
In other words, $S_{i_1} K x_0 \in \mathfrak{B}$, and hence $KS_{i_1}Kx_0 \in K \, \mathfrak{B}$. In a similar way,
$KS_{i_1}Kx_0 \in \mathcal{O}(S_{i_2})$ for some $i_2 \in \{1, \dots ,n \}$, and hence $S_{i_2}KS_{i_1}Kx_0  \in \mathfrak{B}$.
We can continue this process, we derive
$$x_m=S_{i_m}KS_{i_{m-1}}K \dots S_{i_2}KS_{i_1}Kx_0 \in \mathfrak{B}$$
for each $m \geq 1$. Now let $c= \mathrm{max} \{ \| S_i \| : ~i=1, \dots ,n \}$, and recall that each $S_i$ commutes with $K$.
Then we have
$$x_m=(c^{-1} S_{i_{m}})  (c^{-1} S_{i_{m-1}}) \dots  (c^{-1} S_{i_{1}}) (cK)^m x_0 \in \mathfrak{B}.$$
In view of (\ref{M1}) and the fact that all the $c^{-1} S_{i_j}$ have norms at most $1$, we derive
$$\|x_m \|\leq  \| (cK)^m \| \| x_0 \| \to 0,$$
which implies $0 \in \overline{ \mathfrak{B}}$.
But $\mathfrak{B}$ is closed and $0 \notin \mathfrak{B}$, which is a contradiction!
That is, $W_u$ cannot be dense for all $u \neq 0$. This yields a hyperinvariant subspace for $K$.
\end{proof}

\begin{corollary} \label{M292}
Let $ \mathcal{A}$ be a finite dimensional C*-algebra and let $E$ be a
Hilbert $ \mathcal{A}$-module that is not finitely generated.
If $S\in \mathcal{L}(E)$ commutes with a nonzero compact
operator $K$ on $E$, then $S$ has a proper nonzero invariant submodule.
In particular, every nonzero compact operator on a Hilbert
$ \mathcal{A}$-module that is not finitely generated, has a proper nonzero invariant submodule.
\end{corollary}
We close our paper with the following problem which is under study in some special cases.
\begin{problem} \label{problem2}
Let $E$ is a Hilbert $ \mathcal{A}$-module and let $I \neq K \in  \mathcal{K}_{ \mathcal{A}}(E) $.
Characterize those
C*-algebras $ \mathcal{A}$ for which every closed submodule of $E$ is invariant
under $K$.
\end{problem}


\end{document}